\documentclass[10pt]{article}
\usepackage[affil-it]{authblk}
\usepackage{amsmath,amssymb,amsthm,epsfig,color}
\numberwithin{equation}{section}

\usepackage{suffix}
\usepackage{mathtools}
\usepackage[margin=1.5in]{geometry}
\DeclarePairedDelimiterX\MeijerM[3]{\lparen}{\rparen}%
{\begin{smallmatrix}#1 \\ #2\end{smallmatrix}\delimsize\vert\,#3}

\newcommand\MeijerG[8][]{%
  G^{\,#2,#3}_{#4,#5}\MeijerM[#1]{#6}{#7}{#8}}

\WithSuffix\newcommand\MeijerG*[7]{%
  G^{\,#1,#2}_{#3,#4}\MeijerM*{#5}{#6}{#7}}

\newtheorem{theorem}{Theorem}

\newtheorem{remark}[theorem]{Remark}
\newtheorem{definition}[theorem]{Definition}
\newtheorem{lemma}[theorem]{Lemma}
\newtheorem{proposition}[theorem]{Proposition}
\numberwithin{theorem}{section}

\author{Nick Simm\thanks{n.simm@warwick.ac.uk}}
\affil{Mathematics Institute, University of Warwick, Coventry, CV4 7AL, UK}
\date{}

\begin{document}

\title{\Large On the real spectrum of a product of Gaussian random matrices}
\maketitle
\begin{abstract}
Let $X_{m} = G_{1}\ldots G_{m}$ denote the product of $m$ independent random matrices of size $N \times N$, with each matrix in the product consisting of independent standard Gaussian variables. Denoting by $N_{\mathbb{R}}(m)$ the total number of real eigenvalues of $X_{m}$, we show that for $m$ fixed
\begin{equation*}
\mathbb{E}(N_{\mathbb{R}}(m)) = \sqrt{\frac{2Nm}{\pi}}+O(\log(N)), \qquad N \to \infty.
\end{equation*}
This generalizes a well-known result of Edelman et al. \cite{EKS94} to all $m>1$. Furthermore, we show that the normalized global density of real eigenvalues converges weakly in expectation to the density of the random variable $|U|^{m}B$ where $U$ is uniform on $[-1,1]$ and $B$ is Bernoulli on $\{-1,1\}$. This proves a conjecture of Forrester and Ipsen \cite{FI16}. The results are obtained by the asymptotic analysis of a certain Meijer G-function.
\end{abstract}
\section{Introduction}
The subject of non-Hermitian random matrix theory can be said to originate in Ginibre's 1965 paper \cite{Gin65} which introduced three basic random matrix ensembles of interest. These ensembles consist of $N \times N$ matrices of Gaussian variables over the real, complex or quaternion number systems respectively. In the present paper we are exclusively interested in the \textit{real} Ginibre ensemble, which is defined by the probability density function
\begin{equation}
P(G) = \frac{1}{(2\pi)^{N^{2}/2}}\mathrm{exp}\left(-\frac{1}{2}\mathrm{Tr}(GG^{\mathrm{T}})\right) \label{gin}
\end{equation}
acting on the set of all $N \times N$ real matrices\footnote{For ease of presentation we will assume throughout that $N$ is even, the case $N$ odd can be dealt with similarly and does not affect the final results} with respect to the Lebesgue measure. An intriguing feature of the real case is that eigenvalues can be purely real with non-zero probability. Such real eigenvalues were studied by Edelman, Kostlan and Shub \cite{EKS94}, who calculated their expected number and asymptotic density in the limit $N \to \infty$. Since then it was realized that the real eigenvalues form an interesting point process in their own right, with connections to random dynamical systems \cite{FK15}, annihilating and coalescing Brownian motions \cite{TKZ12, TZ11}, and intermediate spectral statistics in quantum chaos \cite{Bee13}. 

The purpose of the present paper is to study the real eigenvalues of the random matrix product $G_{1}\ldots G_{m}$ where for each $i=1,\ldots,m$ the $G_{i}$ are independent and distributed according to \eqref{gin}. Products of random matrices are currently a rapidly evolving field, with many new developments occurring in the last $3$ or $4$ years, see \cite{AI15} for an overview. Our first result answers one of the most basic questions regarding a random product matrix: how many of the eigenvalues are real?
\begin{theorem}
\label{th:asy}
Let $m>0$ be a fixed positive integer. Let $N_{\mathbb{R}}(m)$ denote the number of real eigenvalues of the random product matrix $G_{1}\ldots G_{m}$. Then we have the estimate
\begin{equation}
\mathbb{E}(N_{\mathbb{R}}(m)) = \sqrt{\frac{2Nm}{\pi}}+O(\log(N)), \qquad N \to \infty. \label{numberest}
\end{equation}
\end{theorem}
An obvious yet intriguing feature of \eqref{numberest} is that for $m>1$ \textit{more} eigenvalues congregate onto the real axis. This phenomenon was also discussed recently in the context of the probability that all eigenvalues are real, which was observed to increase monotonically to $1$ as $m \to \infty$ with fixed $N$ \cite{F14,FI16,R17}. These works were motivated by a recent application of the real eigenvalues of products to quantum entanglement \cite{L13}. 

The expected value \eqref{numberest} has been studied numerically in \cite{L13, HJL15, MPT16}, but the value of the pre-factor $\sqrt{2m/\pi}$ is not discussed there and, despite its simplicity, appears to be absent from the literature. Such a concise result is quite surprising given that the spectrum of a product typically depends in a complicated way on the spectra of each individual factor. When $m$ is fixed, the asymptotic order of magnitude in \eqref{numberest} is $\sqrt{N}$, as is the case for $m=1$ \cite{EKS94}. Still for $m=1$, Tao and Vu \cite{TV15} have shown that this persists to independent-entry matrices with more general distributions. The same $\sqrt{N}$ order of magnitude also holds for the expected number of real roots of certain random polynomials \cite{BD97}. The universality of this so-called `$\sqrt{N}-$law' and physical applications are discussed in \cite{Bee13}. 

Next we give a more precise result regarding the global spectral distribution of the real eigenvalues of the product matrix $G_{1}\ldots G_{m}$. Denoting these eigenvalues by $\lambda_{1},\ldots,\lambda_{N_{\mathbb{R}}(m)}$, the \textit{averaged empirical spectral density} is defined by the quantity
\begin{equation}
\rho^{\mathbb{R}}_{N,m}(x) = \mathbb{E}\left(\sum_{j=1}^{N_{\mathbb{R}}(m)}\delta_{N^{-m/2}\lambda_j}(x)\right) \label{meandens}
\end{equation}
and its appropriately normalized version is $h_{N,m}(x) = \frac{\rho^{\mathbb{R}}_{N,m}(x)}{\mathbb{E}(N_{\mathbb{R}}(m))}$. The scaling $N^{-m/2}$ ensures that the eigenvalues remain inside the interval $[-1,1]$ with high probability. Note that $h_{N,m}(x)$ is the probability density function of the random variable $N^{-m/2}\lambda_{U_N}$ where $U_N$ is independent and uniformly distributed on the set $\{1,\ldots,N^{\mathbb{R}}(m)\}$. 
\begin{theorem}
	\label{th:dens}
Choose $\lambda_{N,m}$ uniformly at random from the real eigenvalues of the product matrix $G_{1}G_{2}\ldots G_{m}$. Then $N^{-m/2}\lambda_{N,m} \overset{d}{\to} |U|^{m}B$ as $N \to \infty$, where $U$ is uniform on $[-1,1]$ and $B$ is Bernoulli on $\{-1,1\}$. In other words, for every bounded and continuous function $f$, we have
	\begin{equation}
	\int_{\mathbb{R}}\,f(x)h_{N,m}(x)\,dx \to \frac{1}{2m}\int_{-1}^{1}\,f(x)|x|^{\frac{1}{m}-1}\,dx, \qquad N \to \infty. \label{densconv}
	\end{equation}
\end{theorem}
This result proves a conjecture of Forrester and Ipsen in \cite{FI16}. It states that the distribution of a typical real eigenvalue of the product is asymptotically for large $N$ the same as a typical real eigenvalue of $G_{1}^{m}$, appropriately symmetrized. For the distribution of all real and complex eigenvalues this type of relation between product and power matrices is known due to techniques coming from free probability \cite{BJW10,BNS12} and was generalized to non-Gaussian matrices in \cite{OS11}. However, it is apparently unclear how such techniques can apply to the real spectrum itself or to obtain \eqref{densconv}. 

Our proofs of theorems \ref{th:asy} and \ref{th:dens} follow a unified approach based on the method of moments. Namely, the strategy is to compute the moments
\begin{equation}
M_{k,N}(m) = \int_{\mathbb{R}}x^{k}\rho^{\mathbb{R}}_{N,m}(x)\,dx, \qquad k=0,1,2,3,\ldots \label{moments}
\end{equation}
to leading order in $N$ as $N \to \infty$. In the particular case $k=0$ we obtain $M_{0,N}(m) = \mathbb{E}(N_{\mathbb{R}}(m))$ and consequently \eqref{numberest}. Then it will be proved that
\begin{equation}
\lim_{N \to \infty}\frac{M_{k,N}(m)}{M_{0,N}(m)} = \begin{cases} (mk+1)^{-1} & k \text{ even}\\0 & k \text{ odd}\end{cases} \label{momentconv}
\end{equation} 
which are the moments of the density on the right-hand side of \eqref{densconv}. Since this density is uniquely determined by its moments we establish the weak convergence \eqref{densconv}. This is a popular strategy in random matrix theory and goes back to Wigner who used it to derive the famous semi-circle law for Hermitian random matrices with independent entries. A subtlety here is that one normally obtains the moments by computing the expected traces of powers. Here the problem is that such traces necessarily involve contributions from the complex eigenvalues and thus are not obviously related to \eqref{moments}. For the same reason, techniques coming from free probability do not seem to be of help. Instead our strategy is based on a recent computation of Forrester and Ipsen \cite{FI16}, showing that \eqref{meandens} can be written exactly in terms of Meijer G-functions. Then after obtaining a suitable integral representation for \eqref{moments}, an asymptotic analysis of this exact formula leads to our main results, theorems \ref{th:asy} and \ref{th:dens}.\\\\
\textbf{Acknowledgments}\\
I would like to thank Jesper Ipsen and Peter Forrester for helpful discussions and encouragement. I acknowledge the support of a Leverhulme Trust Early Career Fellowship ECF-2014-309.


\section{Moments and Meijer G-functions}
In this section we review the results for the finite-$N$ density \eqref{meandens} as in \cite{FI16} and use it to give an exact formula for the moments \eqref{moments}. As this necessarily involves working with Meijer G-functions, we begin with a definition. 
\begin{definition}
For a set of real parameters $a_{1},\ldots,a_{p}, b_{1},\ldots,b_{q}$ and $z \in \mathbb{C}$, the \textit{Meijer G-function} is defined by the following contour integral
\begin{equation}
\MeijerG*{m}{n}{p}{q}{a_1, \dots,a_p}{b_1,\dots, b_q}{z} = \frac{1}{2\pi i}\int_{\gamma}\frac{\prod_{j=1}^{m}\Gamma(b_{j}-s)\prod_{j=1}^{n}\Gamma(1-a_j+s)}{\prod_{j=m+1}^{q}\Gamma(1-b_j+s)\prod_{j=n+1}^{p}\Gamma(a_j-s)}\,z^{s}\,ds
\end{equation}
The contour $\gamma$ goes from $-i\infty$ to $i\infty$ with all poles of $\Gamma(b_{j}-s)$ lying to the right of $\gamma$ for all $j=1,\ldots,m$ and all poles of $\Gamma(1-a_{k}+s)$ lying to the left of $\gamma$ for all $k=1,\ldots,n$.
\end{definition}
Functions of Meijer-G type appear very frequently in the study of random matrix products. This might be expected from results for the scalar case: the density function for a product of $m$ independent standard Gaussian variables is proportional to
\begin{equation} w_{m}(x) := \int_{\mathbb{R}^{m}}\prod_{j=1}^{m}dx_{j}\,e^{-x_{j}^{2}/2}\delta(x-x_{1}x_{2}\ldots x_{m}) = \MeijerG*{m}{0}{0}{m}{\line(1,0){15}\vspace{3pt}}{0,\dots, 0}{\frac{x^{2}}{2^{m}}}. \label{weight}
\end{equation}
A simple derivation of \eqref{weight} involves verifying that both sides of the equation have the same Mellin transform. For the matrix products considered here, the function \eqref{weight} plays the same fundamental role that $w_{1}(x) = e^{-x^{2}/2}$ plays in the analysis of a single Ginibre matrix. In this case it is known that the real eigenvalues form a Pfaffian point process, meaning that all $p$-point correlation functions of real eigenvalues can be written as $p \times p$ Pfaffians involving an explicit $2\times 2$ matrix kernel (see \textit{e.g.} \cite{FN07,BS09}). Recently this has been shown to extend to products of random matrices. 
\begin{theorem}[Forrester and Ipsen \cite{FI16}]
	\label{th:FI}
Let $N$ be even. The real eigenvalues of the matrix product $G_{1}\ldots G_{m}$ form a Pfaffian point process with correlation kernel given by
\begin{equation}
\mathbb{K}(x,y) = \begin{pmatrix} D(x,y) & S(x,y)\\ -S(y,x) & I(x,y) \end{pmatrix}
\end{equation}
where 
\begin{align}
S(x,y) &= \sum_{j=0}^{N-2}\frac{w_{m}(x)x^{j}}{(2\sqrt{2\pi}j!)^{m}}(xA_{j}(y)-A_{j+1}(y)),\\
A_{j}(y) &= \int_{\mathbb{R}}w_{m}(v)\mathrm{sgn}(y-v)v^{j}\,dv,\\
D(x,y) &= -\frac{\partial}{\partial y}S(x,y)\\
I(x,y) &= -\int_{x}^{y}S(t,y)\,dt+\frac{1}{2}\mathrm{sgn}(x-y)
\end{align}
In particular, the $p$-point correlation function of the real eigenvalues satisfies
\begin{equation}
\rho_{p}(x_1,\ldots,x_p) = \mathrm{Pf}\,[\mathbb{K}(x_{i},x_{l})_{i,l=1,\ldots,p}] \label{corr}
\end{equation}
\end{theorem}
The results in \cite{FI16} also include correlations involving purely complex eigenvalues, but we will not need them here. From theorem \ref{th:FI}, we can extract the moments \eqref{moments} explicitly.
\begin{lemma}
	\label{lem:moms}
Let $N$ be even. The moments \eqref{moments} satisfy $M_{k,N}(m)=0$ if $k$ is odd, while for $k$ even they are given by $M_{2k,N}(m) = M^{(1)}_{2k,N}(m)-M^{(2)}_{2k,N}(m)$ where
\begin{equation}
\begin{split}
M^{(1)}_{2k,N}(m) &= N^{-mk}\sum_{j=0}^{N/2-1}\frac{2^{(2j+k)m}}{(\sqrt{\pi}(2j)!)^{m}}(a_{j+1,j+k+1}+a_{j+k+1,j+1})\\
M^{(2)}_{2k,N}(m) &= N^{-mk}\sum_{j=0}^{N/2-2}\frac{2^{(2j+1+k)m}}{(\sqrt{\pi}(2j+1)!)^{m}}(a_{j+k+2,j+1}+a_{j+2,j+k+1}) \label{twosums}
\end{split}
\end{equation}
Here $a_{j,k}$ is a particular case of the Meijer G-function 
\begin{equation}
\begin{split}
a_{j,k} &= \MeijerG*{m+1}{m}{m+1}{m+1}{3/2-j, \dots, 3/2-j,1}{0,k, \dots, k}{1} \\
&=\frac{1}{2\pi i}\int_{\gamma}\frac{\Gamma(k-s)^{m}\Gamma(-1/2+j+s)^{m}}{-s}\,ds \label{ajk1}
\end{split}
\end{equation}
where we may take $\gamma = \{-1/4+i\eta : \eta \in \mathbb{R}\}$.
\end{lemma}

\begin{proof}
By theorem \ref{th:FI}, the unscaled density of real eigenvalues is \eqref{corr} with $p=1$ and noting that $I(x,x)=0$, we obtain
\begin{equation}
\rho_1(x) = \sum_{j=0}^{N-2}\frac{w_{m}(x)x^{j}}{(2\sqrt{2\pi}j!)^{m}}(xA_{j}(x)-A_{j+1}(x)) \label{dens}
\end{equation}
where
\begin{equation}
A_{j}(x) = \int_{-\infty}^{\infty}w_{m}(y)\mathrm{sgn}(x-y)y^{j}\,dy.
\end{equation}
Since $w_{m}(x)$ in \eqref{weight} is an even function of $x$ so too is $\rho_{1}(x)$. Therefore all odd moments vanish identically. To compute the even moments, consider the coefficients
\begin{equation}
\alpha_{j,k} := \int_{-\infty}^{\infty}\int_{-\infty}^{\infty}w_{m}(x)w_{m}(y)x^{j-1}y^{k-1}\mathrm{sgn}(y-x)\,dx\,dy.
\end{equation}
We multiply both sides of \eqref{dens} by $x^{2k}$ and integrate over $\mathbb{R}$, leading to
\begin{equation}
M_{2k,N}(m) = N^{-mk}\sum_{j=0}^{N-2}\frac{1}{(2\sqrt{2\pi}j!)^{m}}(\alpha_{j+1,j+2k+2}(m)+\alpha_{j+2k+1,j+2}(m)) \label{moment1stform}
\end{equation}
In obtaining \eqref{moment1stform} we have used that the $2k^{\mathrm{th}}$ moments of the unscaled density $\rho_{1}(x)$ are $N^{-mk}$ times those of $\rho^{\mathbb{R}}_{N,m}(x)$ in \eqref{meandens}. The quantity $\alpha_{j,k}(m)$ is skew-symmetric and depends on the parity of $j$ and $k$. As shown in \cite[Proposition 3]{F14}, we have
\begin{equation}
\alpha_{2j-1,2k}(m) = 2^{(j+k-1/2)m}\MeijerG*{m+1}{m}{m+1}{m+1}{3/2-j, \dots, 3/2-j,1}{0,k, \dots, k}{1}.
\end{equation}

Splitting the sum into even and odd values of $j$ and using the skew-symmetry property gives \eqref{twosums}. 

\end{proof}

We now deduce a useful integral representation for $\alpha_{j,k}$. This appears without proof in \cite{FI16}.
\begin{lemma}
	For $j,k \geq 1$, the coefficients $a_{j,k}$ in \eqref{ajk1} admit the integral representation
	\begin{equation}
	a_{j,k} = \Gamma(j+k-1/2)^{m}\int_{1}^{\infty}\frac{dx_{m}}{x_{m}}\prod_{l=1}^{m-1}\left[\int_{0}^{\infty}\frac{dx_{l}}{x_{l}}\frac{(x_{l}/x_{l+1})^{j-1/2}}{(1+x_{l}/x_{l+1})^{j+k-1/2}}\right]\frac{x_{1}^{k}}{(1+x_{1})^{j+k-1/2}} \label{ajkform}
	\end{equation}
\end{lemma}
\begin{proof}
	This follows from writing the product of Gamma functions in \eqref{ajk1} as an Euler integral
	\begin{equation}
	\left(\frac{\Gamma(k-s)\Gamma(-1/2+j+s)}{\Gamma(j+k-1/2)}\right)^{m} = \int_{\mathbb{R}_{+}^{m}}\prod_{l=1}^{m}dt_{l}\,\frac{t_{l}^{k-s-1}}{(1+t_{l})^{k+j-1/2}}
	\end{equation}
	and interchanging $\prod_{l=1}^{m}dt_{l}$ with $ds$. Then one exploits Perron's formula
	\begin{equation}
	\frac{1}{2\pi i}\int_{\gamma}\frac{u^{-s}}{-s}\,ds = \begin{cases} 0, & 0 < u < 1\\ 1, & u > 1 \end{cases}
	\end{equation}
	after which \eqref{ajkform} follows from a suitable change of variables. Since the integration over $\gamma$ is neither compact nor absolutely integrable, the interchange requires some justification. We restrict $\gamma$ to the bounded region $\gamma_{R} = \{-1/4+i\eta : |\eta| \leq R\}$ where the interchange is justified by Fubini's theorem and obtain
	\begin{equation}
	a_{j,k} = \Gamma(j+k-1/2)^{m}\lim_{R \to \infty}\int_{\mathbb{R}_{+}^{m}}\prod_{l=1}^{m}\frac{t_l^{k-1}}{(1+t_l)^{k+j-1/2}}\int_{\gamma_{R}}\frac{u^{-s}}{-s}\,ds\,dt_{1}\ldots dt_{m} \label{ajk}
	\end{equation}
	where $u := \prod_{l=1}^{m}t_{l}$. Now if $u>1$ we close the $\gamma_{R}$ contour to the right and pick up the pole at $s=0$. The error is given by the integration over the large semi-circular contour $C_{R} := \{-1/4-iRe^{i\phi} : 0 < \phi < \pi\}$. We have
	\begin{equation}
	\begin{split}
	\bigg{|}\int_{C_{R}}\frac{u^{-s}}{-s}\,ds\bigg{|} &\leq 2u^{1/4}\int_{0}^{\pi/2}e^{-R\log(u)\sin(\phi)}d\phi\\
	&\leq Cu^{1/4}\frac{1-u^{-R/2}}{R\log(u)} \label{CRbound}
	\end{split}
	\end{equation}
	If $u>1+\epsilon$ the right hand side of \eqref{CRbound} is uniformly bounded by $u^{1/4}/R$ and inserting into \eqref{ajk} gives zero in the limit $R \to \infty$. If $1 < u < 1+\epsilon$ we change variables $t_{m} = u/(t_{1}\ldots t_{m-1})$ and use the bound $(1+u/t)^{-k-j+1/2} \leq t$ so that the contribution to \eqref{ajk} is bounded by
	\begin{equation}
	\begin{split}
	\int_{\mathbb{R}_{+}^{m-1}}\prod_{l=1}^{m-1}\frac{dt_{l}}{(1+t_{l})^{k+j-1/2}}\int_{1}^{1+\epsilon}u^{-1/4+k-1}\frac{1-u^{-R/2}}{R\log(u)}\,du = O(\log(R)/R)
	\end{split}
	\end{equation}
	Now if $u<1$, we close the contour to the left where the integrand is analytic. Then the only contribution comes from the integral over $\tilde{C}_{R} := \{-1/4+iRe^{i\phi} : 0 < \phi < \pi\}$ which tends to zero by an identical argument. We have
	
	\begin{equation}
	a_{j,k} = \Gamma(j+k-1/2)^{m}\int_{\mathbb{R}_{+}^{m}}\prod_{l=1}^{m}dt_{l}\frac{t_{l}^{k-1}}{(1+t_{l})^{k+j-1/2}}\,1_{t_1 \ldots t_m>1}
	\end{equation}
	and the change of variables $t_{1} = x_{1}$, $t_{i} = x_{i}/x_{i-1}$ for $i=2,\ldots,m$ leads straightforwardly to the representation \eqref{ajkform}.
\end{proof}

\section{Asymptotic analysis}
In this section we prove our main results, theorems \ref{th:asy} and \ref{th:dens}. We begin by showing that both theorems are consequences of lemma \ref{lem:moms} and the following $j \to \infty$ asymptotics for the coefficients in the sums \eqref{twosums}.
\begin{proposition}
\label{prop:ijj}
Let $I_{j,k}(m)$ denote the $m$-fold integral in \eqref{ajkform}, that is $I_{j,k}(m) = a_{j,k}\Gamma(j+k-1/2)^{-m}$. Then as $j \to \infty$
\begin{equation}
I_{j+l_1,j+l_2}(m) = j^{-m/2}4^{-mj}2^{-m(l_1+l_2-2)}\left(a_{0}(m)+a_{1,l_1,l_2}(m)j^{-1/2}+O(1/j)\right) \label{ij1}
\end{equation}
where
\begin{equation}
\begin{split}
&a_{0}(m) = \pi^{m/2}2^{-m/2-1}m^{-1/2}\\
&a_{1,l_1,l_2}(m) = \pi^{(m-1)/2}2^{-m/2-1}(1/2-l_1+l_2)m^{1/2}
\end{split}
\end{equation}
\end{proposition}
Before giving the proof, we show how its conclusion quickly implies the main results of the paper.
\begin{proof}[Proof of theorems \ref{th:asy} and \ref{th:dens}]
First note that we may consider only the contribution to the sums \eqref{twosums} from sufficiently large $j>j_{0}(m,k)$ where the asymptotics \eqref{ij1} hold, since the sum from $j=0$ to $j_{0}$ is $O(N^{-mk})$. We have
\begin{equation}
\begin{split}
M^{(1)}_{2k,N} &= 2N^{-mk}\sum_{j=j_{0}}^{N/2-1}\frac{2^{(2j+k)m}\Gamma(2j+k+3/2)^{m}}{(\sqrt{\pi}(2j)!)^{m}}j^{-m/2}4^{-mj}2^{-mk}\\
&\times (a_{0}(m)+a_{1,0,0}(m)j^{-1/2}+c(j_0)/j)
\end{split}
\end{equation}
Stirling's formula implies 
\begin{equation}
\begin{split}
M^{(1)}_{2k,N} &= N^{-mk}\sum_{j=j_0}^{N/2-1}j^{mk}2^{mk}(m^{-1/2}+j^{-1/2}\pi^{-1/2}m^{1/2}+O(1/j))\\
&= \frac{N}{2}\frac{1}{mk+1}m^{-1/2}+\frac{1}{2}\sqrt{2Nm/\pi}\frac{1}{2mk+1}+O(\log(N))
\end{split}
\end{equation}
A similar computation yields 
\begin{equation}
M^{(2)}_{2k,N} = \frac{N}{2}\frac{1}{mk+1}m^{-1/2}-\frac{1}{2}\sqrt{2Nm/\pi}\frac{1}{2mk+1}+O(\log(N))
\end{equation}
Taking the difference of these two asymptotic estimates gives the moments to leading order in $N$
\begin{equation}
M_{2k,N} = \frac{\sqrt{2Nm/\pi}}{2mk+1}+O(\log(N)) \label{mkasy}
\end{equation}
Setting $k=0$ gives the expected number of real eigenvalues
\begin{equation}
\mathbb{E}(N_{\mathbb{R}}(m)) = \sqrt{\frac{2Nm}{\pi}}+O(\log(N))
\end{equation}
and proves theorem \ref{th:asy}. Normalizing \eqref{mkasy} by $\mathbb{E}(N_{\mathbb{R}}) = M_{0,N}$ shows that the $2k^{\mathrm{th}}$ moment of $h_{N,m}(x)$ converges to $\frac{1}{2mk+1}$ which is the $2k^{\mathrm{th}}$ moment of $|x|^{1/m-1}/(2m)$ on $x \in [-1,1]$. This completes the proof of theorem \ref{th:dens}. 
\end{proof}
\begin{remark}
	A disadvantage of truncating the sum at $j=j_{0}$ is that the constant $O(1)$ term in the asymptotics \eqref{numberest} is left undetermined. Although we also do not compute the $\log(N)$ correction, it seems reasonable it will cancel out of the final results, as we know happens for $m=1$ \cite{EKS94}. However proving this requires looking at the next order term in \eqref{ij1} which is left for future investigation.
\end{remark}
The only remaining task is to deduce the asymptotics \eqref{ij1} of proposition \ref{prop:ijj}. This turns out to be a standard application of the saddle point method for multi-dimensional integrals. However, in this case the analysis is complicated by the fact that the saddle point lies on the boundary of the $m$-dimensional domain of integration. This requires slightly different analysis compared with the typical case of interior points and will lead to an expansion in half integer powers of $j$ instead of integer powers one might normally expect. Hence we give a self-contained exposition for our particular application.

\begin{proof}[Proof of proposition \ref{prop:ijj}]
By definition we have
\begin{align}
I_{j+l_1,j+l_2}(m) &= \int_{1}^{\infty}\frac{dx_{m}}{x_{m}}\prod_{l=1}^{m-1}\left[\int_{0}^{\infty}\frac{dx_{l}}{x_{l}}\frac{(x_{l}/x_{l+1})^{j+l_1-1/2}}{(1+x_{l}/x_{l+1})^{2j+l_1+l_2-1/2}}\right]\frac{x_{1}^{j+l_1}}{(1+x_{1})^{2j+l_1+l_2-1/2}} \notag\\
&= \int_{1}^{\infty}\int_{\mathbb{R}_{+}^{m-1}}\,e^{j\Phi_{m}(\vec{x})}F(\vec{x})dx_{1}\ldots dx_{m} \label{multi}
\end{align}
where
\begin{equation}
\Phi(\vec{x}) = 2\log(x_{1})-2\log(1+x_{1})-\log(x_{m})-2\sum_{l=1}^{m-1}\log(1+x_{l}/x_{l+1})
\end{equation}
and
\begin{equation}
F(\vec{x}) = \frac{x_{1}^{l_1+l_2-3/2}}{x_{m}^{l_1-1/2}(1+x_{1})^{l_1+l_2-1/2}}\prod_{l=1}^{m-1}(1+x_{l}/x_{l+1})^{-l_1-l_2+1/2}x_{l+1}^{-1}
\end{equation}
The saddle point equations for $\Phi$ are $x_{1}^{2}=x_{2}$, $x_{l-1}x_{l+1}=x_{l}^{2}$, $l=2,\ldots,m-1$ and $x_{m-1}=x_{m}$. It is easily seen that the only solution of this set of equations is $\vec{x} = \vec{p}$ where
\begin{equation}
\vec{p} := (1,1,\ldots,1,1). \label{sad}
\end{equation}
The Hessian matrix of $\Phi$ at $\vec{p}$ is
\begin{equation}
H_{m} = \begin{pmatrix}-1 & 1/2 & 0 & 0 & \ldots & 0 & 0\\
1/2 & -1 & 1/2 & 0 & 0 & \ldots & 0\\
0 & 1/2 & -1 & 1/2 & 0 & \ldots & 0\\
\vdots & \vdots & \vdots & \ldots & \vdots & \vdots & \vdots\\
0 & \ldots & 0 & 1/2 & -1 & 1/2 & 0\\
0 & \ldots & 0 & 0 & 1/2 & -1 & 1/2\\
0 & 0 & \ldots & 0 & 0 & 1/2 & -1/2
\end{pmatrix}
\end{equation}
which is clearly negative so that $\vec{p}$ is the unique maximum. It will be useful in what follows to have the formulae
\begin{equation}
\begin{split}
\det(-H_{m}) &= \frac{m+1}{2^{m}}\\
(H_{m}^{-1})_{i,j} &= \frac{2i(j-m-1)}{m+1}, \quad i \leq j
\end{split}
\end{equation}
which follow from known facts regarding the determinant and inverse of a symmetric tri-diagonal matrix.

The analysis begins by localizing the integral near the saddle point, so that 
\begin{equation}
I_{j+l_1,j+l_2} \sim \int_{1}^{1+\epsilon}\int_{[1-\epsilon,1+\epsilon]^{m-1}}e^{j\Phi(\vec{x})}F(\vec{x})dx_{1}\ldots dx_{m} \label{localize}
\end{equation}
up to exponentially small errors. This is justified because away from the unique maximum we have 
\begin{equation}
a:= \sup_{\vec{v} \in [1+\epsilon,\infty) \times ([1-\epsilon,1+\epsilon]^{m-1})^{c}}\Phi(\vec{v}) < \Phi(\vec{p})
\end{equation}
so that the integration over the complement of the region in \eqref{localize} is dominated by
\begin{equation}
e^{aj}\int_{1}^{\infty}\int_{\mathbb{R}_{+}^{m-1}}F(\vec{x})\,dx_{1}\ldots dx_{m} = e^{aj}I_{l_1,l_2}(m) = O(e^{aj})
\end{equation}
which gives rise to an exponentially small contribution for any $\epsilon>0$. By Taylor's theorem, for sufficiently small $\epsilon>0$, we may expand $\Phi(\vec{x})$ near the saddle point. For $\vec{x_{0}} = \vec{x}-\vec{p}$, we write
\begin{equation}
\begin{split}
\Phi(\vec{x}) &= \Phi(\vec{p}) + \frac{1}{2}\vec{x_{0}}^{\mathrm{T}}H_{m}\vec{x_{0}} + T_{3}(\vec{x_{0}},\vec{p})+O(|\vec{x_{0}}|^{4})\\
F(\vec{x}) &= F(\vec{p}) + \vec{x_{0}}^{\mathrm{T}}dF(\vec{p}) + O(|\vec{x_{0}}|^{2}) \label{asyphiF}
\end{split}
\end{equation}
where $T_{3}(\vec{x_{0}},\vec{p})$ is the third order term in the Taylor expansion of $\Phi(\vec{x})$ near $\vec{x}=\vec{p}$. Combined with a simple bound on the exponential function, we obtain from \eqref{asyphiF} the estimate
\begin{equation}
\begin{split}
&e^{j\Phi(\vec{x})}F(\vec{x}) = \mathrm{exp}\left(j\left(\Phi(\vec{p}) + \frac{1}{2}\vec{x_0}^{\mathrm{T}}H_{m}\vec{x_0}\right)\right)\\
&\times \left(F(\vec{p})+\vec{x_0}^{\mathrm{T}}dF(\vec{p}) + jF(\vec{p})T_{3}(\vec{x_0},\vec{p})+ O\left(|\vec{x_0}|^{6}j^{2}e^{\epsilon j|\vec{x_0}|^{2}}+j|\vec{x_0}|^{4}+|\vec{x_0}|^{2}\right)\right) \label{taylor}
\end{split}
\end{equation}
where the error is uniform in $\vec{x} \in [1-\epsilon,1+\epsilon]^{m}$ and all $j>0$. Inserting \eqref{taylor} into the integral \eqref{localize} and changing variables $\vec{u} = j^{1/2}(\vec{x}-\vec{p})$ shows that
\begin{equation}
\begin{split}
&I_{j+l_1,j+l_2} = j^{-m/2}e^{j\Phi(\vec{p})}\int_{0}^{\infty}\int_{\mathbb{R}^{m-1}}e^{\frac{1}{2}\vec{u}^{\mathrm{T}}H_{m}\vec{u}}\\
&\times \left(F(\vec{p})+j^{-1/2}\vec{u}^{\mathrm{T}}dF(\vec{p}) + j^{-1/2}F(\vec{p})T_{3}(\vec{u},\vec{p})\right)d\vec{u} + O(j^{-m/2-1}e^{j\Phi(\vec{p})}) \label{gaussint}
\end{split}
\end{equation}
Note that $e^{j\Phi(\vec{u})} = 4^{-mj}$ which appears in the claimed asymptotics. The quantity $T_{3}(\vec{u},\vec{p})$ can be calculated explicitly 
\begin{equation}
T_{3}(\vec{u},\vec{p}) = \frac{u_{m}^{3}}{4}+\sum_{p=1}^{m-1}\left(\frac{u_{p}^{3}}{2}-\frac{u_{p}^{2}u_{p+1}}{4}-\frac{u_{p}u_{p+1}^{2}}{4}\right) \label{phi3}
\end{equation}
Furthermore
\begin{equation}
F(\vec{p})+\vec{u}^{\mathrm{T}}dF(\vec{p})j^{-1/2} =2^{m/2-m(l_1+l_2)}\left(1-j^{-1/2}\left(\frac{3+2(l_1-l_2)}{4}u_{m}+\sum_{p=1}^{m-1}u_{p}\right)\right) \label{dF1}
\end{equation}
Inserting \eqref{dF1} and \eqref{phi3} into \eqref{gaussint} shows that it remains to compute order of $m$ Gaussian integrals with some low order polynomials in the integrand. All such integrals can be computed in terms of the following generating function
\begin{equation}
F_{p,l}(t,s) = \int_{0}^{\infty}x_{m}^{l}e^{-x_{m}^{2}/4}\int_{\mathbb{R}^{m-1}}\mathrm{exp}\left(\frac{1}{2}\vec{x}^{\mathrm{T}}H_{m-1}\vec{x}+\mu_{p}(t,s)^{\mathrm{T}}\vec{x}\right)d\vec{x} \label{genfun}
\end{equation}
where $\mu_{p}(t,s)$ is a vector entirely zero except for entries $\mu_{p} = t$, $\mu_{p+1}=s$ and $\mu_{m-1} = x_{m}/2$ and here $\vec{x} = (x_{1},\ldots,x_{m-1})$. In terms of the partial derivatives
\begin{equation}
F^{(n,m)}_{p,l} := \frac{\partial^{n+m}}{\partial t^{n}\partial s^{m}}F_{p,l}(t,s)\bigg{|}_{t=s=0}. \label{partial}
\end{equation}
we can rewrite the terms in the expansion \eqref{gaussint} in terms of the generating function \eqref{genfun} as follows
\begin{equation}
\begin{split}
&I_{j+l_1,j+l_2} = j^{-m/2}e^{j\Phi(\vec{p})}F(\vec{p})F_{0,0}(0,0)\\
&+j^{-(m+1)/2}e^{j\Phi(\vec{p})}F(\vec{p})\left(\sum_{p=1}^{m-2}\left((1/2)F^{(3,0)}_{p,0}-(1/4)F^{(2,1)}_{p,0}-(1/4)F^{(1,2)}_{p,0}-F^{(1,0)}_{p,0}\right)\right.\\
&\left.+(1/2)F^{(3,0)}_{m-1,0}-(1/4)F^{(2,0)}_{m-1,1}-(1/4)F^{(1,0)}_{m-1,2}-\frac{3+2(l_1-l_2)}{4}F^{(0,0)}_{0,1}\right.\\
&\left.-F^{(1,0)}_{m-1,0}+\frac{F_{0,3}(0,0)}{4}\right)+O(j^{-m/2-1}e^{j\Phi(\vec{p})}). \label{Ijfinal}
\end{split}
\end{equation}
It remains to compute the generating function and the required derivatives. Standard facts about Gaussian integration allow the integral in \eqref{genfun} over $\mathbb{R}^{m-1}$ to be done explicitly. We obtain $F_{p,l}(t,s) =K_{m}G_{p,l}(t,s)I_{p,l}(t,s)$
where
\begin{equation}
\begin{split}
I_{p,l}(t,s) &= \int_{0}^{\infty}z^{l}e^{-z^{2}/m+tzp/m+sz(p+1)/m}\,dz\\
G_{p,l}(t,s) &= \mathrm{exp}\left(\frac{t^{2}(m-p)+s^{2}(p+1)(m-(p+1))+2tsp(m-(p+1))}{m}\right)
\end{split}
\end{equation}
and
\begin{equation}
K_{m} = \sqrt{\frac{(2\pi)^{m-1}}{\det(-H_{m-1})}} = 2^{m-1}\pi^{(m-1)/2}m^{-1/2}
\end{equation}
Consequently, the derivatives \eqref{partial} to any order can be computed explicitly (most easily in a computer algebra package such as MAPLE). Inserting the results of the computation into \eqref{Ijfinal} and using the formula $e^{j\Phi(\vec{p})} = 4^{-mj}$ gives the result stated in the proposition.

\end{proof}

\bibliography{products}
\bibliographystyle{plain}
\end{document}